\documentclass[12pt,reqno]{amsart}

\usepackage[margin=1in]{geometry}

\title[Global well-posedness for a slightly supercritical SQG equation]{Global well-posedness for a slightly supercritical surface quasi-geostrophic equation}
\date{\today}

\author{Michael Dabkowski}
\address{Department of Mathematics, University of Wisconsin Madison, 480 Lincoln Dr., Madison, WI 53706}
\email{dabkowsk@math.wisc.edu}

\author{Alexander Kiselev}
\address{Department of Mathematics, University of Wisconsin Madison, 480 Lincoln Dr., Madison, WI 53706}
\email{kiselev@math.wisc.edu}

\author{Vlad Vicol}
\address{Department of Mathematics, The University of Chicago, 5734 University Ave., Chicago, IL 60637}
\email{\tt vicol@math.uchicago.edu}

\usepackage{amsfonts,amsmath,latexsym,amssymb,verbatim,amsbsy,times,color}
\usepackage{amsthm}
\usepackage{pstricks}

\usepackage{hyperref}

\theoremstyle{plain}
\newtheorem{theorem}{Theorem}[section]
\newtheorem{definition}[theorem]{Definition}
\newtheorem{lemma}[theorem]{Lemma}
\newtheorem{proposition}[theorem]{Proposition}

\theoremstyle{definition}
\newtheorem{remark}[theorem]{Remark}

\numberwithin{equation}{section}

\renewcommand\hat{\widehat}

\def\RR{{\mathbb R}}
\def\TT{{\mathbb T}}
\def\DD{\mathcal D}

\def\LL{\Lambda}
\def\omb{\omega_B}

\def\zz{\zeta}

\begin{document}


\begin{abstract}
We use a nonlocal maximum principle to prove the global existence of smooth solutions for a slightly supercritical surface quasi-geostrophic equation. By this we mean that the velocity field $u$ is obtained from the active scalar $\theta$ by a Fourier multiplier with symbol $i \zz^\perp |\zz|^{-1} m(|\zz|)$, where $m$ is a smooth increasing function that grows slower than $\log \log |\zz|$ as $|\zz|\rightarrow \infty$.
\end{abstract}


\dedicatory{Dedicated to Peter Constantin on occasion of his 60th birthday}



\maketitle

\section{Introduction} \label{sec:intro}

The surface quasi-geostrophic equation (SQG) has recently been a
focus of research efforts by many mathematicians. It is probably
the simplest physically motivated evolution equation of fluid
mechanics for which, in the supercritical regime, it is not known
whether solutions stay regular or can blow up. The equation is
given by
\begin{align}
&\partial_t \theta + (u \cdot \nabla) \theta + \LL^\alpha \theta = 0,\,\,\,\,\theta(\cdot,0) = \theta_0 \nonumber \\
& u = \nabla^\perp \LL^{-1} \theta  \nonumber 
\end{align}
on $(x,t) \in \TT^2 \times [0,\infty)$, where $\LL =
(-\Delta)^{1/2}$. The SQG equation appeared in the mathematical
literature for the first time in \cite{CMT}, and since then has
attracted significant attention, in part due to certain similarities with
three dimensional Euler and Navier-Stokes equations.
The equation has $L^\infty$
maximum principle \cite{R,CC}, which makes the $\alpha =1$
dissipation critical. It has been known since \cite{R} that the
equation has global smooth solutions (for appropriate initial
data) when $\alpha
>1.$ The global regularity in the critical case has been settled
independently by Kiselev-Nazarov-Volberg \cite{KNV} (in the
periodic setting) and Caffarelli-Vasseur \cite{CV} (in the whole
space as well as in the local setting). A third proof of the same
result was provided recently in \cite{KN2}. All these proofs are quite
different. The method of \cite{CV} is inspired by DeGiorgi
iterative estimates, while the approach of \cite{KN2} uses appropriate
set of test functions and estimates on their evolution. The method
of \cite{KNV}, on the other hand, is based on a new technique
which can be called a nonlocal maximum principle. The idea is to
prove that the evolution \eqref{EQ:1} preserves a
certain modulus of continuity $\omega$ of the solution. The control is
strong enough to give a uniform bound on $\|\nabla
\theta\|_{L^\infty}$ in the critical case, which is sufficient for
global regularity.

In the supercritical case, the only results available so far (for large initial data) have
been on conditional regularity and finite time regularization of solutions. For instance, it was shown by Constantin and Wu~\cite{CW} that if the solution is $C^{1-\alpha}$, then it is smooth. Finite time
regularization has been proved by Silvestre \cite{Silvestre} for $\alpha$ sufficiently close to $1$, and for the whole dissipation
range $0<\alpha<1$ by Dabkowski \cite{D} (with an alternative
proof of the latter result given in \cite{K2}). The issue of
global regularity in the case $\alpha\in(0,1)$ remains an outstanding open problem.

Our goal here is to advance global regularity very slightly into the supercritical regime for the SQG equation. For technical reasons (and inspired by \cite{CCW}),
it is more convenient for
us to introduce supercriticality in the velocity $u$ rather than in the dissipation. Namely,
let $m(\zz) = m(|\zz|)$ be a smooth, radial, non-decreasing function on $\RR^2$, such that $m(\zz)\geq 1$ for all $\zz \in \RR^2$. We shall consider the active scalar equation,
\begin{align}
&\partial_t \theta + (u \cdot \nabla) \theta + \LL \theta = 0,\,\,\,\,\theta(\cdot,0) = \theta_0  \label{EQ:1}\\
& u = \nabla^\perp \LL^{-1} m(\LL) \theta \label{EQ:2}
\end{align}
on $(x,t) \in \TT^2 \times [0,\infty)$, where $m(\LL) \theta$ is defined by its Fourier transform $(m(\LL)\theta)^{\widehat{}}(\zz) = m(\zz) \hat{\theta}(\zz)$. Note that $m \equiv 1$ gives us the usual critical SQG equation. We shall consider symbols $m(\zz)$ which for all sufficiently large $|\zz|$  satisfy the growth condition
\begin{align}
&\lim_{|\zz|\rightarrow \infty} \frac{m(\zz)}{\ln \ln|\zz|} = 0\label{eq:m:cond:1}.
\end{align}
In addition we require that
\begin{align}
\lim_{|\zz|\rightarrow \infty} \frac{|\zz| m'(\zz)}{m(\zz)} = 0 \label{eq:m:cond:2}
\end{align}
and that the symbol $m$ is of H\"ormander-Mikhlin type, i.e., there exists $C>0$ such that
\begin{align}
|\zz|^{k} | \partial_{\zz}^{k} m(\zz) | \leq C m(\zz) \label{eq:m:cond:3}
\end{align}
holds for all $\zz \neq 0$, and all $k \in \{0,\ldots,d+2\}$.  The main result of this paper is:
\begin{theorem}[\bf Slightly supercritical SQG]
\label{thm:SQG}
Assume that $\theta_0 \in C^\infty(\TT^2)$. If the symbol $m$ satisfies \eqref{eq:m:cond:1}--\eqref{eq:m:cond:3}, then there exists a unique global $C^\infty$ smooth solution $\theta$ of \eqref{EQ:1}--\eqref{EQ:2}.
\end{theorem}

\begin{remark} \label{rem:conditions:1}
The condition \eqref{eq:m:cond:2} can be improved to require only $\lim_{|\zz|\rightarrow \infty} |\zz| m'(\zz) / m(\zz) <1$,
but is adapted here for the sake of simplicity.
\end{remark}

The result we prove here is reminiscent of the slightly
supercritical Navier-Stokes regularity result of Tao \cite{Tao1}.
The challenge in the SQG case is that while regularity for
critical Navier-Stokes is easy to prove by energy method, there is
no similarly simple proof of regularity for the critical SQG. The
criticality of the SQG equation is controlled by the $L^\infty$
norm, and the order of differentiation is the same in the
nonlinearity and dissipation term. This makes global regularity
for large data surprising at the first look. All three proofs of global regularity
for critical SQG are somewhat subtle and involved.
Scaling plays a crucial role in all existing proofs. The
main contribution of this paper is to show that one can advance,
at least a little, beyond the critical scaling.

To prove Theorem~\ref{thm:SQG}, we rely on the original method of
\cite{KNV}. This method is based on constructing a modulus of
continuity $\omega(\xi)$, Lipshitz at zero and growing at
infinity, which is respected by the critical SQG evolution: if the
initial data $\theta_0$ obeys $\omega,$ so does the solution
$\theta(x,t)$ for every $t>0.$ By scaling, in the critical regime any rescaled modulus
$\omega_B(\xi) = \omega(B\xi)$ is also preserved by the evolution.
This allows, given smooth initial data $\theta_0,$ to find $B$
such that $\theta_0$ obeys $\omega_B$ and thus, due to
preservation of $\omega_B,$ gain sufficient control of solution
for all times. The unboundedness of $\omega$ is crucial for this
argument; applying it with bounded $\omega$ would correspond to
controlling only initial data of limited size. It appears that the
maximal growth of $\omega$ one can afford in the critical SQG case
is a double logarithm, dictated by balance of nonlinear and
dissipative term estimates. The idea of the proof of
Theorem~\ref{thm:SQG}, and the key observation of this paper, is
that it is possible to trade some of this growth in $\omega$ for
a slightly rougher velocity $u$ (or, likely, slightly weaker
dissipation). In the process, one loses critical scaling, but the argument
can be made to work by manufacturing a family
of moduli $\omega_B$ preserved by the evolution which are no longer
a single rescaled modulus.

We anticipate that the approach we develop here will have other
applications. In particular, it can be applied to a slightly
supercritical Burgers equation. In this case, one can expect to prove global
regularity for a more singular equation, supercritical by almost a
logarithmic multiplier. This is due to the existence of moduli
with logarithmic growth conserved by the evolution. Consideration
of the Burgers equation, as well as applications to modified SQG,
and the case of supercritical dissipation is postponed to a subsequent
publication \cite{DKV}.

\section{Preliminaries} \label{sec:prelim}

The local and conditional regularity for the SQG-type equations is by now standard. 
In particular, we have
\begin{proposition}[\bf Local existence of smooth solution] \label{prop:LWP}
Given $\theta_0 \in H^s(\TT^2)$, for some $s>1$, there exists
$T>0$ and a solution $\theta(\cdot,t) \in C([0,T],H^s) \cap
C^\infty((0,T]\times \TT^2)$ of \eqref{EQ:1}--\eqref{EQ:2}.
Moreover, the solution may be continued as a smooth solution
beyond $T$ as long as $ \Vert \nabla
\theta\Vert_{L^1(0,T;L^\infty(\TT^2))} < \infty$.
\end{proposition}
The proof of a similar result with standard SQG velocity and critical or supercritical dissipation can be found, for example, in \cite{Dong}.
The addition of the multiplier $m$ into $u$ does not create any essential difficulties in the argument.

\begin{definition}[\bf Modulus of continuity] \label{def:MOC}
We call a function $\omega \colon (0,\infty) \rightarrow
(0,\infty)$ a modulus of continuity if $\omega$ is increasing,
continuous, concave, piecewise $C^2$ with one sided derivatives,
and it additionally satisfies $\omega'(0+) = \infty$ or
$\omega''(0+) = -\infty$. We say that a smooth function $f$ 
obeys the modulus of continuity $\omega$ if $|f(x) - f(y)| <
\omega(|x-y|)$ for all $x\neq y$.
\end{definition}
We recall that if $f\in C^\infty(\TT^2)$ obeys the modulus $\omega$, then $\Vert \nabla f \Vert_{L^\infty} < \omega'(0)$. In addition, observe that a function $f\in C^\infty(\TT^2)$ automatically obeys any modulus of continuity $\omega(\xi)$ that lies above the function $\min\{\xi \Vert \nabla f \Vert_{L^\infty}, 2 \Vert f \Vert_{L^\infty} \}$.

We will construct a family of moduli of continuity $\omega_B$ that will be preserved by the evolution. To prove this nonlocal maximum principle, we will use the following
outline. The proofs of Lemmas~\ref{break} and \ref{lemma:d:MOC} below can be found in \cite{KNV}.
\begin{lemma}[\bf Breakthrough scenario]\label{break}
Assume $\omega$ is a modulus of continuity such that
$\omega(0+)=0$ and $\omega''(0+)= -\infty.$ Suppose that the
initial data $\theta_0$ obeys $\omega.$ If the solution
$\theta(x,t)$ violates $\omega$ at some positive time, then there must
exist $t_1>0$ and $x \neq y \in \TT^2$ such that
\[ \theta(x,t_1) - \theta(y,t_1) = \omega(|x-y|), \]
and $\theta(x,t)$ obeys $\omega$ for every $0 \leq t<t_1.$
\end{lemma}
Let us consider the breakthrough scenario for a modulus $\omega.$ A simple computation shows that
\begin{align}
\partial_{t} \left( \theta(x,t) - \theta(y,t) \right)|_{t=t_1} &= u \cdot \nabla \theta (y,t_1) - u \cdot \nabla \theta (x,t_1) + \LL\theta(y,t_1) - \LL\theta(x,t_1)\notag\\
&\leq |u(x,t_1)-u(y,t_1)| \omega'(\xi) +  \LL\theta(y,t_1) - \LL\theta(x,t_1)\label{eq:MOC}.
\end{align}
If we can show that the expression in \eqref{eq:MOC} must be strictly negative, we obtain a contradiction: $\omega$ cannot be broken, and hence it is preserved.
To estimate \eqref{eq:MOC} we need
\begin{lemma}[\bf Modulus of continuity for the drift velocity]\label{lemma:u:MOC}
Assume that $\theta$ obeys the modulus of continuity $\omega$, and that the drift velocity is given as $u =  \nabla^\perp \LL^{-1} m(\LL) \theta$. Then $u$ obeys the modulus of continuity $\Omega$ defined as
\begin{align}
\Omega(\xi) =  A \left( \int_{0}^{\xi} \frac{\omega(\eta)m( \eta^{-1})}{\eta} d\eta
+ \xi \int_{\xi}^{\infty} \frac{\omega(\eta)m( \eta^{-1})}{\eta^2} d\eta\right) \label{eq:u:moc}
\end{align}
for some positive constant$A\geq 1$  that only depends on the function $m$.
\end{lemma}
The proof of Lemma~\ref{lemma:u:MOC} shall be given in the
Appendix. For the dissipative terms, we have:
\begin{lemma}[\bf Dissipation control]\label{lemma:d:MOC}
Assume we are in a breakthrough scenario as in Lemma~\ref{break}. Then
\begin{align}
\LL\theta(y,t_1) - \LL\theta(x,t_1) \leq \DD(\xi) & \equiv  \frac{1}{\pi}  \int_{0}^{\xi/2} \frac{\omega(\xi+2\eta)
 + \omega(\xi - 2\eta) - 2 \omega(\xi)}{\eta^2} d\eta\notag\\
 &\qquad + \frac{1}{\pi} \int_{\xi/2}^{\infty} \frac{\omega(2\eta+\xi) -
 \omega(2\eta - \xi) - 2\omega(\xi)}{\eta^2} d\eta. \label{eq:dissipation:moc}
\end{align}
\end{lemma}
Given the three Lemmas above and \eqref{eq:MOC}, in order to verify the
preservation of $\omega$ for all time, it is sufficient to check that
$\Omega(\xi)\omega'(\xi)+\DD(\xi) <0$ for every $\xi>0.$

The conditions imposed on the symbol $m$ have two consequences which shall be useful later:
\begin{lemma}[\bf Further properties of $\boldsymbol{m}$] \label{lemma:m}
Let $m$ be smooth radial radially non-decreasing function which satisfies \eqref{eq:m:cond:2}.
Then it holds that
\begin{align}
\lim_{|\zz| \rightarrow \infty} \frac{m(\zz) (1 + \ln |\zz|)}{|\zz|} = 0 \label{eq:m:supercritical}
\end{align}
and we have
\begin{align}
\int_{0}^{|\zz|} m(r^{-1}) dr \leq 2 |\zz| m(|\zz|^{-1}) \label{eq:m:IntBound}
\end{align}
for all $|\zz|$ which are sufficiently small, depending on $m$.
\end{lemma}
\begin{proof}[Proof of Lemma~\ref{lemma:m}]
From \eqref{eq:m:cond:2} it follows that there exists $r_{0} > 0 $ such that for all $|\zz| \geq r_{0}$ we have $2 |\zz| m'(\zz) \leq m(\zz)$.  In order to prove \eqref{eq:m:supercritical},
let $f(r) = r m(r)$. We have $f'(r) \leq 3 m(r)/2 = (3/2) r^{-1} f(r) $, for all $r\geq r_{0}$, and hence $f(r) \leq f(r_{0}) r_{0}^{-3/2} r^{3/2} = m(r_{0}) r_{0}^{-1/2} r^{3/2}$. Therefore $r^{-1} (1 + \ln r) m(r) = r^{-2} (1+\ln r) f(r) \leq (1+\ln r) r^{-1/2} m(r_{0}) r_{0}^{-1/2}\rightarrow 0$ as $r \rightarrow \infty$. 
In fact, it is easy to see that $m(\zz)/|\zz|^a \rightarrow 0$ as $\zz \rightarrow \infty$ for every $a>0,$ but we will not need this stronger bound in the proof. 

To prove \eqref{eq:m:IntBound}, we note that the function $r^{1/2} m(r^{-1})$ is non-decreasing on $r \leq r_{0}^{-1}$. Therefore $m(r^{-1}) \leq |\zz|^{1/2} m(|\zz|^{-1}) r^{-1/2}$, and \eqref{eq:m:IntBound} follows if $|\zz| \leq r_{0}^{-1}$, by integrating in $r$.
\end{proof}

\section{Proof of Main Theorem} \label{sec:moduli}
The main difference between our argument here and \cite{KNV} is
that since \eqref{EQ:1}--\eqref{EQ:2} is beyond the critical
scaling, one cannot use $\omb(\xi) = \omega(B \xi)$ to
construct the needed family of moduli of continuity, from a fixed
modulus $\omega$.

\subsection{A suitable family of moduli of continuity}
We fix a sufficiently small positive constant $\kappa$, to be chosen later.
For $B\geq 1$, we define $\delta = \delta(B)$ to be the unique solution of
\begin{align}
B \delta m(\delta^{-1}) = \kappa. \label{eq:delta}
\end{align}
To see that $\delta$ exists and is unique, let $g(\delta) = \delta m(\delta^{-1})$. Due to \eqref{eq:m:supercritical}, we have that $g(\delta) \rightarrow 0$ as $\delta\rightarrow 0+$, and due to \eqref{eq:m:cond:2}, we have that $g'(\delta) =  m(\delta^{-1}) - \delta^{-1} m'(\delta^{-1}) \geq m(\delta^{-1}) /2 > 0$, for all $\delta \leq r_{0}^{-1}$. So $g$ is increasing (and continuous) at least until $r_{0}^{-1}$, and hence if $\kappa$ is chosen such that $\kappa \leq g(r_{0}^{-1}) = r_{0}^{-1}m(r_{0})$, since $B \geq 1$, the equation $g(\delta) = \kappa B^{-1}$ will have a unique solution. Note that $\delta(B) \rightarrow 0$ as $B \rightarrow \infty$ since $g(0+)=0$, and $\delta(B)$ is a strictly decreasing function of $B$.

Having defined $\delta(B)$ for each $B\geq 1$, we shall consider the modulus of continuity $\omb$ defined as the continuous function with $\omb(0) = 0$ and
\begin{align}
  \omb'(\xi) &= B - \frac{B^{2}}{8\kappa} \xi m(\xi^{-1}) \left(4 + \ln \frac{\delta(B)}{\xi} \right), \ \mbox{for all } 0 < \xi \leq \delta(B)\label{smallom}\\
  \omb'(\xi) &= \frac{\gamma}{\xi (4 + \ln(\xi/ \delta(B))) m(\delta(B)^{-1})}, \ \mbox{for all } \xi > \delta(B)\label{farom}
\end{align}
where $\kappa = \kappa(A,m)$ and $\gamma = \gamma(\kappa,A,m)$ are sufficiently small constants to be chosen
later. To verify that $\omb$ is a modulus of continuity, we need to check monotonicity, concavity,  that $ 0 < \omb'(0+) < \infty$, and that $\omb''(0+) = - \infty$.

From \eqref{eq:m:supercritical} we know that $\xi m(\xi^{-1}) (1 + \ln |\xi|^{-1}) \rightarrow 0$ as $\xi \rightarrow 0+$, and therefore, for every $B \geq 1$ we have that $\omb'(0+) = B$. Note that in fact we have $\omb'(\xi) < B$, and hence $\omb(\xi) \leq B\xi$ for all $0< \xi \leq \delta(B)$. Taking the derivative of \eqref{smallom} we obtain
\begin{align}
\omb''(\xi) & = -\frac{B^{2}}{8\kappa} \left( \left( m(\xi^{-1}) - \xi^{-1} m'(\xi^{-1}) \right) \left(4 + \ln \frac{\delta(B)}{\xi} \right) - m(\xi^{-1}) \right) \notag\\
&\leq -\frac{B^{2}}{8\kappa} \left( \frac 12 m(\xi^{-1}) \left(4 + \ln \frac{\delta(B)}{\xi} \right) - m(\xi^{-1}) \right)\notag\\
& \leq -\frac{B^{2}}{32\kappa} m(\xi^{-1}) \left( 4 + \ln \frac{\delta(B)}{\xi} \right) \label{eq:m:second:derivative}
\end{align}
which implies that $\omb''(\xi) \rightarrow - \infty$ as $\xi \rightarrow 0+$ since $m(\xi^{-1}) \geq 1$ for all $\xi >0$. Note that in the first inequality of \eqref{eq:m:second:derivative} we have used $2 \xi^{-1} m'(\xi^{-1}) \leq m(\xi^{-1})$, which holds for all $\xi \leq \delta(B)$ as long as $\delta(B)$ is sufficiently small (how small it needs to be depends only on $m$). 
One can always ensure that $\delta(B)$ is small enough since $\delta(B) \leq \delta(1)$ for all $B \geq 1$, and $\delta(1)$ can be made arbitrarily small by choosing $\kappa$ to be sufficiently small, depending only on $m$.

From \eqref{farom} and \eqref{eq:m:second:derivative} it is clear that the concavity of $\omb$ may only fail at $\xi = \delta(B)$. However, from \eqref{eq:delta} and \eqref{smallom} we obtain
    \begin{align}
    \omb'(\delta(B)-) = B - \frac{B^{2}}{2 \kappa} \delta(B) m(\delta(B)^{-1}) &= \frac{B}{2}\label{eq:concave-}.
    \end{align}
    On the other hand by \eqref{farom} we have
    \begin{align}
      \omb'(\delta(B)+) = \frac{\gamma}{4 \delta(B) m(\delta(B)^{-1})} = \frac{\gamma B}{4 \kappa} < \frac{B}{4}\label{eq:concave+}
    \end{align}
    for all $\gamma < \kappa$. Together, \eqref{eq:concave-} and \eqref{eq:concave+} show that $\omb$ is concave.

To prove that $\omb$ is monotonically increasing, it  is sufficient to verify that $\omb'(\xi) > 0$ for all $0< \xi  < \delta(B)$. But $\omb'(0)=B \geq 1$ and $\omb'(\xi)$ is decreasing on $(0,\delta(B))$ due to \eqref{eq:m:second:derivative}, so that we only need to verify that $\omb'(\delta(B)-) > 0$. This follows directly from \eqref{eq:concave-}.


Let us denote $\Omega_B(\xi)$ and $\DD_B(\xi)$ respectively the
modulus of the velocity $u$ given by \eqref{eq:u:moc} and
dissipation estimate \eqref{eq:dissipation:moc} corresponding to
$\omb(\xi).$ It is sufficient to prove two things: that each initial data $\theta_{0}$ obeys some modulus of continuity $\omb$ for a suitable $B \geq 1$, and that the expression in \eqref{eq:MOC} when
computed for each $\omb$ is strictly negative for all $\xi>0.$ 

\subsection{Modulus of continuity for the initial data}
First we show that any initial data $\theta_{0} \in C^{\infty}(\TT^{2})$ obeys a modulus of continuity $\omb$ for some sufficiently large $B.$ 
As noted earlier, this is achieved if we find a sufficiently large $B$ such that $\omb(\xi) > \min\{\xi \Vert \nabla \theta_{0} \Vert_{L^\infty}, 2 \Vert \theta_{0} \Vert_{L^\infty} \}$ for all $\xi > 0$. Observe that due to concavity of $\omb$ it is sufficient to find $B$ such that
\[ \omega_B\left(\frac{2\|\theta_0\|_{L^\infty}}{\|\nabla \theta_0\|_{L^\infty}}\right) \geq 2\|\theta_0\|_{L^\infty}.\]
However, note that for every fixed $a>0,$ we have $a > \delta(B)$ if $B$ is sufficiently large, and
\begin{align*}
\int_{\delta(B)}^a \frac{\gamma}{\xi (4+\ln(\xi/\delta(B))) m(\delta(B)^{-1})} d\xi = \frac{\gamma}{m(\delta(B)^{-1})} \ln(1+\ln(a /\delta(B)) )\rightarrow \infty
\end{align*}
as $B \rightarrow \infty$ due to our assumption
\eqref{eq:m:cond:1} on growth of $m$, and since $\delta(B) \rightarrow 0$ as $B \rightarrow \infty$.
This shows that any smooth $\theta_{0}$  obeys a modulus of continuity $\omb$ if $B$ is chosen large enough.

\subsection{Conservation of the modulus of continuity}
We shall now prove that if $\kappa$ is chosen sufficiently small
(depending only on $m,$ and $A$), and $\gamma$ is chosen sufficiently
small (depending only on $\kappa,m,$ and $A$), then the expression
\eqref{eq:MOC} is strictly negative, i.e. $\Omega_B(\xi)
\omb'(\xi) + \DD_B(\xi) < 0$, for all $\xi>0$. Note that neither
$\kappa$, nor $\gamma$ will depend on $B \geq 1$.

\subsection*{The case $0<\xi \leq \delta(B)$} We first observe that $\omb'(\xi) \leq B$ for all $\xi \in (0,\delta(B)]$. Using concavity of $\omega$ and the mean value theorem we may estimate
    \begin{align*}
    \DD_B(\xi) \leq \frac{1}{\pi} \xi \omb''(\xi).
    \end{align*}
    In addition, it follows from the bound \eqref{eq:m:second:derivative} on $\omb''(\xi)$ and the above estimate that
    \begin{align}
	\DD_B(\xi) \leq  - \frac{1}{32\pi \kappa} B^{2} \xi m(\xi^{-1}) \left( 4 + \ln \frac{\delta(B)}{\xi} \right). \label{eq:I:dissipation}
    \end{align}
    The main issue is to estimate the contribution from $\Omega_B(\xi)$. From \eqref{eq:u:moc} and \eqref{smallom} we have that
    \begin{align}
      \Omega_B(\xi)\omb'(\xi) &\leq A B \left( B\int_0^\xi m(\eta^{-1}) d\eta + B\xi \int_{\xi}^{\delta(B)} \frac{m(\eta^{-1})}{\eta} d\eta + \xi \int_{\delta(B)}^\infty \frac{\omb(\eta)  m(\eta^{-1}) }{\eta^2} d\eta \right)\notag\\
      &\leq A B \left( 2B \xi m(\xi^{-1}) + B\xi m( \xi^{-1}) \ln  \frac{\delta(B)}{\xi} + \xi  m(\xi^{-1}) \int_{\delta(B)}^\infty \frac{\omb(\eta)}{\eta^2} d\eta \right) \label{eq:I:drift:1}.
    \end{align}
    In the second inequality of \eqref{eq:I:drift:1} we have used the monotonicity of $m$ and the inequality \eqref{eq:m:IntBound},
    which holds for all $\xi \leq \delta(B)$, whenever $\delta(B)$ is sufficiently small, depending only on $m$. But note that letting $\kappa$ be sufficiently small, depending on $m$ and  not on $B$, we ensure that $\delta(1)$ is sufficiently small, and the bound $\delta(B) \leq \delta(1)$ for all $B \geq 1$, justifies the applicability of \eqref{eq:m:IntBound} for all $B \geq 1$.

    In order to estimate $\int_{\delta(B)}^\infty\omb(\eta)/\eta^2 d\eta$, we integrate by parts and use the slow growth of $\omb$ (cf.~\eqref{eq:m:cond:1}) to obtain
    \begin{align}
      \int_{\delta(B)}^\infty \frac{\omb(\eta)}{\eta^2} d\eta &\leq \frac{\omb(\delta(B))}{\delta(B)} + \int_{\delta(B)}^{\infty} \frac{\gamma}{\eta^2 (4 + \ln(\eta / \delta(B)) ) m(\delta(B)^{-1})} d\eta \notag\\
      &\leq B + \frac{\gamma}{4 \delta(B) m(\delta(B)^{-1})} = B + \frac{\gamma B }{4 \kappa} \leq 2 B \label{eq:I:drift:2}
    \end{align}
    if $\gamma < \kappa$, since $m(\delta(B)^{-1})\geq 1$. Combining \eqref{eq:I:dissipation} with \eqref{eq:I:drift:1} and \eqref{eq:I:drift:2} we obtain
    \begin{align}
      \Omega_B(\xi)\omb'(\xi) + \DD_B(\xi) &\leq \left(A - \frac{1}{32 \pi \kappa} \right) B^2 \xi m(\xi^{-1}) \left( 4 + \ln \frac{\delta(B)}{\xi}\right) < 0  \label{eq:I:drift:3}
    \end{align}
    for all $\xi \in (0,\delta(B)]$ if we choose $\kappa$ so that $32 \pi \kappa A < 1$. To avoid a circular argument, note that $\kappa$ was chosen independently of $\gamma$ and $B$,  it only depends on $m$ and $A$.

\subsection*{The case $\xi > \delta(B)$} We observe that for each $B\geq 1$ the modulus of continuity $\omb$ satisfies
\begin{align}
\omb(2\xi) \leq \frac 32 \omb(\xi), \ \mbox{for all }\xi\geq
\delta(B).\label{eq:conv}\end{align} Indeed due to the definition \eqref{farom}
of $\omega_B$, we have
\begin{align*}
\omb(2\xi) \leq \omb(\xi) + \frac{\gamma}{4 m(\delta(B)^{-1})}
\end{align*}
for every $\xi \geq \delta(B).$ But from the monotonicity of $\omb$ and the mean value theorem we have $\omb(\xi) \geq
\omb(\delta(B)) \geq \delta(B) \omb'(\delta(B)-)$, since $\omb'$ is strictly decreasing on $(0, \delta(B))$. By \eqref{eq:delta} and \eqref{eq:concave-} it follows that taking
$
\gamma < \kappa
$
is sufficient for \eqref{eq:conv} to hold.
Using \eqref{eq:conv}, we estimate
\begin{align}
\DD_{B}(\xi) \leq \frac{1}{\pi} \int_{\xi/2}^{\infty} \frac{\omega_B(2\eta+\xi) - \omega_B(2\eta - \xi) - \omega_B(2\xi)-\frac12 \omega_B(\xi)}{\eta^2}
d\eta
\leq -\frac{1}{2\pi} \frac{\omb(\xi)}{\xi} \label{eq:II:1}
\end{align}
which holds for all $\xi > \delta(B)$. Next, let us bound the term
arising from $\Omega_{B}(\xi) \omb'(\xi) $ in \eqref{eq:MOC},
namely
\begin{align}
A \omb'(\xi) \left( \int_{0}^{\xi} \frac{\omb(\eta)m(\eta^{-1})}{\eta} d\eta + \xi \int_{\xi}^{\infty} \frac{\omb(\eta)m(\eta^{-1})}{\eta^2} d\eta\right) .\label{eq:II:2}
\end{align}
We first estimate
\begin{align}
\int_{0}^{\xi} \frac{\omb(\eta)m( \eta^{-1})}{\eta} d\eta &\leq B \int_{0}^{\delta(B)} m(\eta^{-1}) d\eta + \int_{\delta(B)}^{\xi}\frac{\omb(\eta)m( \eta^{-1})}{\eta} d\eta \notag\\
&\leq 2 B\delta(B) m(\delta(B)^{-1}) + \omb(\xi) m( \delta(B)^{-1}) \ln \frac{\xi}{\delta(B)}\notag\\
& = 2 \kappa + \omb(\xi) m( \delta(B)^{-1}) \ln \frac{\xi}{\delta(B)}\label{eq:II:3}
\end{align}
for all $\xi > \delta(B)$. Above we used \eqref{eq:delta} and \eqref{eq:m:IntBound}, which may be applied since $\delta(B)$ is sufficiently small with respect to $m$ for any $B \geq 1$.
Furthermore, upon integration by parts we have
\begin{align}
\xi \int_{\xi}^{\infty} \frac{\omb(\eta)m( \eta^{-1})}{\eta^2} d\eta &\leq \xi m( \xi^{-1}) \int_{\xi}^{\infty} \frac{\omb(\eta)}{\eta^2} d\eta \notag\\
& = \xi m(\delta(B)^{-1}) \left( \frac{\omb(\xi)}{\xi} + \frac{\gamma}{m(\delta(B)^{-1})} \int_{\xi}^{\infty} \frac{1}{\eta^{2} (4 + \ln (\eta/\delta(B)))} d\eta \right) \notag\\
&\leq \omb(\xi) m(\delta(B)^{-1}) + \gamma .  \label{eq:II:5}
\end{align}
Therefore, inserting the bounds \eqref{eq:II:3} and \eqref{eq:II:5}  into \eqref{eq:II:2}, and letting $\gamma \leq \kappa$, we obtain
\begin{align}
\Omega_{B}(\xi) \omb'(\xi) &\leq A \omb'(\xi) \left( \gamma + 2\kappa + \omb(\xi) m(\delta(B)^{-1}) \left( 1 + \ln \frac{\xi}{\delta(B)} \right) \right) \notag\\
&\leq \frac{A  \gamma}{\xi (4 + \ln \xi/\delta(B) ) m(\delta(B)^{-1})}  \left( 3\kappa + \omb(\xi) m(\delta(B)^{-1}) \left( 1 + \ln \frac{\xi}{\delta(B)} \right) \right) \notag\\
&\leq  \frac{2 A \gamma \omb(\xi) }{ \xi} \label{eq:II:6}
\end{align}
since $\kappa \leq 2 \omb(\delta(B)) m(\delta(B)^{-1}) \leq 2 \omb(\xi) m(\delta(B)^{-1})$. Indeed, the latter holds since as above we have
\begin{align*}
\omb(\delta(B)) \geq \delta(B) m'(\delta(B)-) = \frac{B \delta(B)}{2} = \frac{\kappa}{2 m(\delta(B)^{-1})}.
\end{align*}
Lastly, from \eqref{eq:II:1} and \eqref{eq:II:6} it follows that
\begin{align}
\Omega_{B}(\xi) \omb'(\xi) + \DD_{B}(\xi) < \left(2 A \gamma - \frac{1}{2\pi}\right) \frac{\omb(\xi)}{\xi} < 0
\end{align}
as long as $\gamma$ is chosen small enough so that $4 \pi A \gamma <1 $.

\section{Appendix} 
\label{sec:app} Here we give details regarding the proof of
Lemma~\ref{lemma:u:MOC}. Let $m(\zz)$ be a continuous, radial,
non-decreasing function on $\RR^{d}$, smooth on $\RR^{d}$, with
$m(\zz) = m(|\zz|)\geq 1$ for all $\zz \in \RR^{d}$. Assume that
$m(\zz)$ satisfies the H\"ormander-Mikhlin-type condition
(cf.~\cite{Stein})
\begin{align}
|\zz|^{k} | \partial_{\zz}^{k} m(\zz) | \leq C m(\zz) \label{ap:cond:1}
\end{align}
for some $C \geq 1$, all $\zz \neq 0$, and all $k \in \{0,\ldots,d+2\}$. In addition we require that
\begin{align}
\lim_{|\zz|\rightarrow \infty} \frac{|\zz| m'(\zz)}{m(\zz)}  = 0. \label{ap:cond:2}
\end{align}
The following lemma gives estimates on the distribution $K$ whose
Fourier transform is $i \zz_{j} |\zz|^{-1} m(\zz)$, for any $j \in
\{1,\ldots,d\}$.

\begin{lemma}[\bf Kernel estimate] \label{lemma:Kernel}
Let $K(x)$ be the kernel of the operator $\partial_{j}
\LL^{-1}m(\LL)$, where $m$ is smooth on $\RR^d$, radial,
non-decreasing in radial variable, and satisfies the conditions
\eqref{ap:cond:1}--\eqref{ap:cond:2}. Then we have
\begin{align}
|K(x)| \leq C |x|^{-d} m(|x|^{-1}) \label{ap:K:1}
\end{align}
and
\begin{align}
|\nabla K(x)| \leq C |x|^{-d-1} m(|x|^{-1}) \label{ap:K:2}
\end{align}
for all $x \ne 0\in \RR^d.$ 
\end{lemma}
\begin{proof}[Proof of Lemma~\ref{lemma:Kernel}]
Consider a smooth non-increasing radial cutoff function $\eta(\zz)
= \eta(|\zz|)$ which is identically $1$ on $|\zz|\leq 1/2$, and
vanishes identically on $|\zz|\geq 1$. For $R>0$, let
$\eta_{R}(|\zz|) = \eta(|\zz|/R)$. Then, for $R>0$ to be chosen
later, we decompose
\begin{align*}
K(x) = C \int_{\RR^{d}} \eta_{R}(\zz) m(\zz) i \zz_{j}|\zz|^{-1} e^{i \zz \cdot x} d\zz + C \int_{\RR^{d}} (1-\eta_{R}(\zz))  m(\zz) i \zz_{j} |\zz|^{-1} e^{i \zz \cdot x} d\zz = K_{1}(x) + K_{2}(x).
\end{align*}
Since $m(\zz)$ is increasing, and $\eta_{R}$ is supported on $B_{R}$, we may bound $|K_{1}(x)| \leq C R^{d} m(R)$. On the other hand, upon integrating by parts $d+2$ times,  using \eqref{ap:cond:1} and the fact that $\partial_{\zz} (1-\eta_{R}(\zz))$ is supported on $R/2 \leq |\xi| \leq R$, we obtain
\begin{align}
|K_{2}(x)| &\leq C |x|^{-d-2} \int_{\RR^{d}} \left| \partial_{\zz}^{d+2} \left( (1-\eta_{R})(\zz) m(\zz)
i \zz_{j} |\zz|^{-1} \right) \right| d\zz \notag\\
&\leq C |x|^{-d-2} \left(  R^{-d-2} \int_{R/2\leq |\zz| \leq R}  m(\zz) d\zz + \int_{|\zz|\geq R/2}
|\zz|^{-d-2} m(\zz) d\zz \right).\label{ap:K2}
\end{align}
Observe that condition \eqref{ap:cond:2} shows there exists some
$r>0$ such that for all $|\zz| \geq r$ we have $2 |\zz|
m'(\zz) \leq m(\zz)$, and hence the function $|\zz|^{-1/2}
m(|\zz|)$ is non-increasing for $|\zz| \geq r$. Consider first
small $x,$ $|x| \leq 1/2r.$ Letting $R = |x|^{-1}$, we have that
$R/2 \geq r$. Using the facts that $m(|\zz|)$ is
non-decreasing, and $|\zz|^{-1/2} m(|\zz|)$ is non-increasing on
$|\zz| \geq r$, we obtain
\begin{align}\label{k2}
|K_{2}(x)| &\leq C |x|^{-d} m(|x|^{-1})
\end{align}
which upon recalling the earlier bound on $K_{1}$ concludes the
proof of \eqref{ap:K:1} for small $x$. For $|x|\geq 1/2r,$ we
can set $R=1$ and obtain that
\[ |K_2(x)| \leq C|x|^{-d-2}, \]
since due to \eqref{ap:cond:2} and the continuity of $m$ we have $|m(\zz)|\leq
C(m)|\zz|^{1/2}.$ On the other hand,
\[ K_1(x) = C \int_{\RR^d} (c_0 i \zz_j|\zz|^{-1}+
\varphi(\zz))e^{i\zz \cdot x} \, d\zz, \] where $c_0$ is a
constant and $\varphi(\zz) \in C_0^\infty.$ This gives the bound
\[ |K(x)| \leq C|x|^{-d}, \]
which together with \eqref{k2} implies \eqref{ap:K:1} for $|x|\geq
1/2r$. The bounds for $\nabla K(x)$ are obtained in the same fashion, the
only difference being an extra factor of $\zz$ in the estimates.
\end{proof}


Having estimated the kernel of the operator $\theta \mapsto u$, we are now ready to estimate the modulus of continuity of the velocity $u$, in terms of the modulus of continuity of the active scalar $\theta$.

\begin{proof}[Proof of Lemma~\ref{lemma:u:MOC}]
The proof is similar to that of \cite[Lemma]{KNV}. Fix $x\neq y$, and let $\xi = |x-y|$. Since $u = \nabla^{\perp} \left( \LL^{-1} m(\LL) \theta \right)$ we have that $\int_{|x|=1} K(x) d\sigma(x) = 0$, and hence we may bound
\begin{align*}
u(x) - u(y) & = \int_{|x-z|\leq 2\xi} K(x-z) (\theta(z) - \theta(x)) dz  - \int _{|y-z|\leq 2\xi} K(y-z) (\theta(z) - \theta(y)) dz \notag\\
& \qquad + \int_{|x-z| \geq 2\xi} K(x-z) (\theta(z) - \theta(\bar z)) dz  - \int _{|y-z|\geq 2\xi} K(y-z) (\theta(z) - \theta(\bar z)) dz
\end{align*}
where the integrals are taken in the principal value sense, and $\bar{z} = (x+y)/2$. Using the estimates on the kernel $K$ from Lemma~\ref{lemma:Kernel}, 
we obtain
\begin{align}
|u(x) - u(y)| &\leq C \int_{0}^{2\xi} \frac{m(\eta^{-1}) \omega(\eta)}{\eta}d\eta  + \int_{|\bar z -z| \geq 3\xi} | K(x-z) - K(y-z)|  |\theta(z) - \theta(\bar z)|  dz \notag\\
&\qquad  +  \int _{3\xi/2 \leq |\bar z -z|\leq 3\xi} \left( | K(x-z) | + | K(y-z)| \right)  | \theta(z) - \theta(\bar z) | dz \label{eq:u:moc:1}.
\end{align}
To estimate the second integral on the right hand side, note that
for $|z - \bar z|\geq 3\xi$, by the mean value theorem and
\eqref{ap:K:2}, we have
\begin{align*}
| K(x-z) - K(y-z)| \leq C \xi |z- \bar z|^{-3} m(|z-\bar z|^{-1}).
\end{align*}
Here we use that $m(s r) \leq s^{C} m(r)$ holds by
\eqref{ap:cond:1} for $s>1$. The third integral on the right hand
side of \eqref{eq:u:moc:1} is bounded using \eqref{ap:K:1} and we
obtain
\begin{align}
|u(x) - u(y)| &\leq C \int_{0}^{3\xi} \frac{m(\eta^{-1}) \omega(\eta)}{\eta}d\eta  + C \xi \int_{3\xi}^{\infty} \frac{m(\eta^{-1}) \omega(\eta)}{\eta^{2}} d\eta \label{eq:u:moc:2}
\end{align}
for all $\xi \neq 0$. The final result then follows from
\eqref{eq:u:moc:2} using the concavity of $\omega$ and the
monotonicity of $m$.
\end{proof}

\noindent {\bf Acknowledgement.} \rm
MD was supported in part by the NSF grant DMS-0800243.
AK acknowledges support of the NSF grant DMS-1104415, and thanks Terry Tao for a useful discussion.

\end{document}